\documentclass{amsart}

\usepackage{amsmath,amssymb,amsthm}

\newtheorem{prop}{Proposition}
\newtheorem{thm}[prop]{Theorem}
\newtheorem{lem}[prop]{Lemma}

\theoremstyle{definition}

\newtheorem*{defn}{Definition}

\newtheorem*{rem}{Remark}

\newtheorem*{ack}{Acknowledgements}

\def\co{\colon\thinspace}

\newcommand{\C}{\mathbb{C}}
\newcommand{\CP}{\mathbb{C}\mathrm{P}}

\newcommand{\rmd}{\mathrm{d}}

\newcommand{\rme}{\mathrm{e}}

\newcommand{\rmi}{\mathrm{i}}

\newcommand{\wtOm}{\widetilde{\Omega}}

\newcommand{\R}{\mathbb{R}}

\newcommand{\wtW}{\widetilde{W}}

\newcommand{\Z}{\mathbb{Z}}

\DeclareMathOperator{\Int}{Int}

\begin{document}

\author[H.~Geiges]{Hansj\"org Geiges}
\address{Mathematisches Institut, Universit\"at zu K\"oln,
Weyertal 86--90, 50931 K\"oln, Germany}
\email{geiges@math.uni-koeln.de}
\author[K.~Zehmisch]{Kai Zehmisch}
\address{Mathematisches Institut, WWU M\"unster,
Einstein\-stra\-{\ss}e 62, 48149 M\"unster, Germany}
\email{kai.zehmisch@uni-muenster.de}

\title{Cobordisms between symplectic fibrations}

\date{}

\begin{abstract}
We discuss the existence and non-existence of cobordisms
between symplectic surface bundles over the circle.
\end{abstract}

\subjclass[2010]{53D35; 32Q65, 57R17, 57R90}

\thanks{H.~G.\ and K.~Z.\ are partially supported by DFG grants
GE 1245/2-1 and ZE 992/1-1, respectively.}

\maketitle


\section{Introduction}
In \cite{elia04} Eliashberg showed that any weak filling of a contact
$3$-manifold can be embedded into a closed symplectic $4$-manifold.
His result is based on the construction of a symplectic cap.
One part of that cap is a cobordism from the given
contact manifold to a symplectic surface bundle over~$S^1$.
Such cobordisms also play a role in~\cite{dgz14}, which led us
to look at the rigidity or flexibility inherent in the
construction of cobordisms between symplectic fibrations.

\begin{defn}
A \textbf{symplectic fibration} is a pair $(V,\omega)$ consisting
of a closed, connected, oriented $3$-manifold $V$ fibred over the circle
$S^1$ and a closed $2$-form $\omega$ on $V$ that restricts to an area form on
each fibre. The \textbf{genus} of the fibration is the genus of
its fibre.

Let $\Sigma$ be a fibre of the symplectic fibration~$(V,\omega)$.
The \textbf{fibre area}
\[ \int_{\Sigma}\omega=\bigl\langle[\omega],[\Sigma]\bigr\rangle,\]
because of its homological interpretation,
is independent of the choice of fibre.
\end{defn}

Our first result is a rigidity statement for spherical fibrations.

\begin{thm}
\label{thm:non-ex}
Let $(W,\Omega)$ be a compact, connected symplectic $4$-manifold
such that each boundary component $(V,\Omega|_{TV})$ is
a symplectic fibration. If one of these surface bundles
has spherical fibres, then so do all the others.
\end{thm}

In particular, there can be no symplectic round handle
construction for symplectic fibrations, in contrast with
the situation for contact manifolds \cite{adac14,durs13}.

For cobordisms between surface bundles of genus greater than zero,
on the other hand, we shall see that the situation is completely flexible;
for cobordisms between sphere bundles the fibre area is the
only cobordism invariant.
In order to formulate the results, we need to be a little more specific
about the concept of cobordisms between symplectic fibrations,
cf.~\cite{elia04}. We prefer to speak simply of `cobordisms' rather than
`symplectic cobordisms', since the latter is commonly used for cobordisms
inducing contact structures on the boundary.

The $2$-form $\omega$ on a symplectic fibration
$(V,\omega)$ defines an orientation of the fibres of~$V$, and the
orientation of $V$ then determines an orientation on the $1$-dimensional
\emph{characteristic foliation} of $V$ determined by the kernel of~$\omega$,
which is transverse to the fibres. Since $\omega$ is closed, any flow along
this characteristic foliation preserves $\omega$ by Cartan's
formula for the Lie derivative. This allows one to define
an area-preserving \emph{holonomy diffeomorphism} from one
fixed fibre to itself, which determines the bundle up to a
fibre-preserving diffeomorphism.

\begin{defn}
Let $(V^i,\omega^i)$, $i=0,1$, be two symplectic fibrations.
A \textbf{cobordism} from $(V^0,\omega^0)$ to $(V^1,\omega^1)$
is a compact, connected symplectic $4$-manifold $(W,\Omega)$ with
\[ \partial (W,\Omega)=(V^1,\omega^1)\sqcup (-V^0,\omega^0).\]
\end{defn}

\begin{rem}
Here $-V^0$ stands for $V^0$ with the reversed orientation.
By the expression for $\partial (W,\Omega)$ we mean that
with $W$ oriented by $\Omega\wedge\Omega$, the
oriented boundary $\partial W$ equals $V^1\sqcup-V^0$, and
$\Omega|_{TV^i}=\omega^i$ for $i=0,1$.

Recall that our symplectic fibrations are, by definition, connected
manifolds. This assumption is made merely for convenience.
The cobordism-theoretic results below have fairly straightforward analogues
when we allow several boundary components at one or the
other end of the cobordism.
\end{rem}

According to the symplectic neighbourhood theorem for hypersurfaces
\cite[Exercise~3.36]{mcsa98}, the restriction $\Omega|_{TV^i}$
determines $\Omega$ in a neighbourhood of~$V^i$. In fact,
up to symplectomorphism, $\Omega$ looks like
$\omega^i+\rmd (t\alpha^i)$ on a neighbourhood
of~$V^i$, where $\alpha^i$ is a $1$-form that
does not vanish on the characteristic foliation, and $t$
is the collar parameter.

This makes the above notion of cobordism both reflexive and transitive,
but it is not clear, a priori, that it is symmetric ---
observe that the holonomy diffeomorphisms of $(V,\omega)$
and $(-V,\omega)$ are inverses of each other. However, symmetry
of the cobordism relation is one of the consequences of
the following flexibility result.

\begin{thm}
\label{thm:ex}
(a) Any two symplectic fibrations of respective genus $g,g'>0$
are cobordant.

(b) Two spherical symplectic fibrations are cobordant if and only if they
have the same fibre area.
\end{thm}

Symplectic fibrations are a special case of odd-symplectic manifolds
in the sense of Ginzburg~\cite{ginz92}. So our results
are in some sense dual to his. Ginzburg considers cobordisms
carrying odd-symplectic forms inducing given symplectic 
forms on the boundary, whereas here we deal with cobordisms
carrying symplectic forms inducing symplectic fibrations on the
boundary. One could investigate the more general cobordism relation
where the boundary condition is weakened to odd-symplectic,
also in higher dimensions.

Our cobordisms are always supposed to be connected and to have two
non-empty boundary components of opposite orientations. Without that
assumption, any symplectic fibration would be null-cobordant
(in either direction) by~\cite{elia04}, and hence any two symplectic
fibrations would be cobordant by a disconnected cobordism.
\section{Proof of Theorem~\ref{thm:non-ex}}
Let $(W,\Omega)$ be a compact symplectic $4$-manifold satisfying the
assumptions of Theorem~\ref{thm:non-ex}. By the results in Section~3
of \cite{elia04}, the boundaries of $W$ can be capped off, i.e.\
$(W,\Omega)$ embeds symplectically into a closed symplectic $4$-manifold
$(\wtW,\wtOm)$.

Arguing by contradiction, we assume that $(W,\Omega)$ has a
spherical symplectic fibration as one boundary
component, and a further boundary component which
is a symplectic bundle over $S^1$ with fibre a closed, orientable surface
$\Sigma_g$ of genus $g>0$. Then $(\wtW,\wtOm)$ contains a
symplectically embedded copy $S$ of $S^2$ and a symplectically
embedded copy $\Sigma$ of $\Sigma_g$, where $S$ is disjoint from $\Sigma$,
and each has self-intersection number zero.

We next construct two further symplectic $4$-manifolds containing
a symplectically embedded $S^2$ or $\Sigma_g$, respectively, with
self-intersection number zero. By fibre connected sum (also called `symplectic
sum') in the sense of Gompf~\cite{gomp95} we shall then build a symplectic
$4$-manifold with contradictory properties.
\subsection{The building blocks}
Our first building block will
be a symplectic $4$-manifold with a single convex boundary component
and a symplectically embedded surface $\Sigma'$ of some genus $g'>0$
with self-intersection number zero.

Start with a compact symplectic $4$-manifold with two
boundary components $M,M'$, both of which are supposed to be strongly
convex boundaries, so that they carry induced contact structures $\xi,\xi'$
cf.\ \cite[Chapter~5]{geig08}.
Examples of such manifolds have been constructed by McDuff \cite{mcdu91}
and the first author~\cite{geig95}. The
contact structure $\xi'$ is supported (in the sense of
Giroux~\cite{giro02}) by an open book of some genus $g'$;
by stabilising the open book, if necessary, we may assume $g'\geq 1$.
(In fact, Etnyre~\cite{etny04} has shown that any contact structure
induced on a boundary component of a symplectic $4$-manifold
with disconnected convex boundary can only be supported by
an open book of genus at least~$1$; the argument that follows would
likewise produce a contradiction if $g'$ were zero.)
Eliashberg's capping construction~\cite{elia04} applied to
the boundary component $(M',\xi')$ then produces a symplectic $4$-manifold
$(W_{g'},\Omega_{g'})$ with convex boundary $(M,\xi)$. This symplectic
manifold contains, inside the cap, a symplectically
embedded $\Sigma'$ --- resulting from capping off a page ---
having the desired properties.

Our second building block is described in the following proposition.

\begin{prop}
\label{prop:X}
Given $g,g'>0$, there is a closed symplectic $4$-manifold
$(X_g^{g'},\Omega_g^{g'})$ containing disjoint symplectically embedded
copies of $\Sigma_g$ and $\Sigma_{g'}$, each of self-inter\-section zero.
\end{prop}

\begin{proof}
Start with the product $\Sigma_g\times\Sigma_{g'}$, equipped with
a product symplectic structure. Let $T\subset \Sigma_g\times\Sigma_{g'}$
be a Lagrangian torus given as the product of homotopically
non-trivial curves in $\Sigma_g$ and $\Sigma_{g'}$, so that
the homology class $[T]$ is non-trivial in $H_2(\Sigma_g\times\Sigma_{g'})$.
As explained in \cite[Lemma~1.6]{gomp95}, there is a symplectic
form on $\Sigma_g\times\Sigma_{g'}$ for which $T$ is symplectic;
see also the proof of Lemma~\ref{lem:trivial} for more details.
This new symplectic form can be chosen arbitrarily close to the product
form we started with; this allows us to assume that
$\Sigma_g\times\{*\}$ and $\{*\}\times\Sigma_{g'}$ are still symplectic
surfaces.

Now take two copies of this manifold, and perform a symplectic
sum as in \cite{gomp95} along the two copies of~$T$, which have
zero self-intersection. This
produces a closed symplectic $4$-manifold $(X_g^{g'},\Omega_g^{g'})$
containing disjoint symplectically embedded copies of $\Sigma_g$ and
$\Sigma_{g'}$, coming from a surface $\Sigma_g\times\{*\}$ in the first
summand and a surface $\{*\}\times\Sigma_{g'}$ in the second summand;
we only have to ensure that the respective point $*$ is chosen away from
the curves that define~$T$.
\end{proof}

Notice that for the symplectic summing we always assume implicitly
that the symplectic forms on the two summands have been scaled
such that they induce area forms on the relevant surfaces
of equal total area; in this situation, the construction
from \cite{gomp95} is applicable.
\subsection{The symplectic manifold $(W',\Omega')$}
\label{subsection:contradiction}
We define $(W',\Omega')$ as the symplectic $4$-manifold
obtained by symplectically
summing $(\wtW,\wtOm)$, $(X_g^{g'},\Omega_g^{g'})$ and $(W_{g'},\Omega_{g'})$,
where the sum is taken along $\Sigma\subset\wtW$ and
$\Sigma_g\subset X_g^{g'}$, as well as along $\Sigma_{g'}\subset X_g^{g'}$
and $\Sigma'\subset W_{g'}$.

Observe that $(W',\Omega')$ has convex boundary $(M,\xi)$, and
it contains a symplectically embedded $2$-sphere $S$ of self-intersection
zero in the $(\wtW,\wtOm)$ summand. McDuff \cite[Theorem~5.1]{mcdu91}
showed that such a symplectic manifold cannot exist, cf.~\cite{geze13}.
By analysing the moduli space of holomorphic
spheres in $W'$ --- with respect to an almost complex structure $J$
tamed by $\Omega'$ for which $S$ is holomorphic and $M$ is $J$-convex ---
one would find a holomorphic sphere through every point on a path
joining $S$ to~$M$, contradicting the maximum principle at the
convex boundary~$M$.

This contradiction proves Theorem~\ref{thm:non-ex}.
\subsection{An alternative argument}
\label{subsection:alt}
One can prove Theorem~\ref{thm:non-ex} via a more direct route,
at the cost of quoting the deep results of McDuff~\cite{mcdu90}
on the classification of ruled symplectic $4$-manifolds.
Wendl~\cite{wend14} has written detailed lecture notes on
McDuff's work and subsequent developments.

Consider the manifold pair $(\wtW,S)$ constructed at the beginning
of the proof of Theorem~\ref{thm:non-ex}, together with the
symplectic surface $\Sigma\subset\wtW\setminus S$ of genus $g>0$.
By blowing down
any potential exceptional spheres in $\wtW\setminus S$, i.e.\
symplectic spheres of self-intersection~$-1$, we obtain
a minimal pair $(\wtW_0,S)$ in the sense of~\cite{mcdu90}.

Blowing down an exceptional sphere $E$ amounts to taking
a fibre connected sum of $(\wtW,E)$ with $(\CP^2,\CP^1)$.
Write $\nu E$ for a closed tubular neighbourhood of $E$ in $\wtW$;
likewise we write $\nu\CP^1$. Then blowing down $E$
means that we replace $\nu E$ by the $4$-ball
$D^4=\CP^2\setminus\Int(\nu\CP^1)$. The $S^1$-fibres of
$\partial(\nu\CP^1)$ are the Hopf fibres on the boundary
of the complementary $4$-ball $D^4$.
Hence, if $\Sigma$ intersects
$E$ transversely, then the effect on $\Sigma$ of blowing down $E$
is to replace the disjoint discs $\nu E\cap\Sigma\subset\Sigma$
by the discs in $D^4$ bounded by the Hopf fibres $\partial(\nu E)\cap\Sigma$.
We write $\Sigma^*\subset\wtW_0$ for the transformed surface
(the `proper transform').

Observe that $\Sigma^*$ has the same genus as $\Sigma$, but
its self-intersection number will have changed.
Any ordered pair of points in $\Sigma\cap E$ (including pairs made up
of twice the same point) will add $\pm 1$ to the self-intersection number
of~$\Sigma^*$, depending on whether the intersection points
have the same sign or not. This follows from the observation
that any two Hopf fibres bound holomorphic discs that intersect
positively in a single point.

We claim that the blow-down can be arranged in such a way that
there is an almost complex structure on $\wtW_0$,
regular for the class~$[S]$, with respect to
which $\Sigma^*$ is an immersed holomorphic curve.
The proof of \cite[Proposition~4.1]{mcdu90},
cf.\ \cite[Theorem~6.1]{wend14},
shows that $S$ is a fibre in a holomorphic ruling of $\wtW_0$.
The surface $\Sigma^*$ has to intersect one of the spherical fibres
geometrically, but the homological intersection is zero, since $\Sigma^*$
is disjoint from the fibre~$S$. By positivity of intersection,
$\Sigma^*$ would have to coincide with a spherical fibre, contradicting
the assumption that $\Sigma$ (and hence $\Sigma^*$) have positive genus.

It remains to prove the claim. Choose an $\wtOm$-compatible
almost complex structure~$J$
on $\wtW$ for which $\Sigma$ and $S$ are holomorphic, and which
is integrable near these two surfaces. By automatic transversality, see
\cite[Corollary~3.3.4]{mcsa04} or \cite[Theorem~2.27]{wend14},
this $J$ will be regular for the
homology class~$[S]$. As nicely explained in \cite[Theorem~5.1]{wend14},
one can now isotope each exceptional sphere in $\wtW\setminus S$
via symplectic spheres to a unique $J$-holomorphic one.
The curve $\Sigma$ cannot coincide with any exceptional sphere.
Hence, by positivity of intersection, this means in particular
that the intersections of the exceptional spheres with $\Sigma$
are isolated and all count positively.

Now make $J$ integrable near these exceptional spheres.
Then $\Sigma$ descends to an immersed holomorphic curve in~$\wtW_0$.
This completes the proof.

\vspace{2mm}

Beware that, in general, the proper transform $\Sigma^*$ of
a symplectic surface $\Sigma$ will not be symplectic.
Here is an example (with $\Sigma$
of genus~$0$). Start with the symplectic manifold
$\CP^2\#\overline{\CP}^2$, the blow-up of $\CP^2$ in a single point.
Let $E$ be the exceptional divisor and $\Sigma$ a transverse symplectic
copy of~$E$, so that the intersection numbers are $\Sigma\bullet E=-1$
and $\Sigma\bullet\Sigma=-1$. After blowing down~$E$, we have
$\Sigma^*\subset\CP^2$ with self-intersection number~$0$,
so it cannot be realised as a symplectic submanifold in $\CP^2$
for cohomological reasons. Indeed, if we tried to carry
out the argument above, after making $\Sigma$ holomorphic and
isotoping $E$ to its unique holomorphic representative, the two curves
would coincide.

\subsection{Broken Lefschetz fibrations and gauge theory}

The following two observations were made by \.Inan\c{c} Baykur.

\vspace{1mm}

(1) A \emph{near-symplectic structure} on a $4$-manifold is a closed
$2$-form satisfying $\omega^2\geq 0$ and a certain transversality
condition along the zero set of $\omega^2$, see~\cite{bayk09}.
In contrast with Theorem~\ref{thm:non-ex},
one can always find a near-symplectic cobordism between any two symplectic
fibrations, no matter what the fibre genera are. When the fibre areas of
the two symplectic fibrations are the same,
a cobordism can be provided by a broken Lefschetz fibration
in the sense of~\cite{bayk09} (originally introduced in
\cite{adk05} under the name `singular Lefschetz fibration'). In the
case of different symplectic areas one needs to appeal, in addition,
to Lemma~\ref{lem:trivial} below.

\vspace{1mm}

(2) Here is an alternative proof of Theorem~\ref{thm:non-ex},
which relies on Seiberg--Witten theory. Suppose $(W,\Omega)$ were a
compact symplectic $4$-manifold whose boundary components are symplectic
fibrations, with at least one of them spherical and one of higher genus.
Cap off all boundary components.
By taking further fibre connected sums, as in our proof of
Theorem~\ref{thm:non-ex}, one can ensure
that the resulting symplectic $4$-manifold
has Betti number $b_2^+>1$. Now take the fibre sum of two copies of this
manifold along two copies of a symplectic sphere
in the respective cap coming from a spherical
fibration. The resulting manifold would split along a copy
of $S^2\times S^1$. However, according to~\cite[Lemma~15]{bafr},
a closed $4$-manifold with $b_2^+>1$ that splits along
$S^2\times S^1$ into two manifolds with $b_2^+>0$ cannot be symplectic.
This contradiction proves that no such manifold $(W,\Omega)$ exists.

\subsection{Confoliations}
\label{subsection:confol}

Yet another proof of Theorem~\ref{thm:non-ex}
has been suggested by Fran Presas. Again one starts from
a supposed counterexample $(W,\Omega)$ to Theorem~\ref{thm:non-ex}.
By the theory of confoliations, notably \cite[Theorem~2.4.1]{elth98},
the boundary components which are non-spherical
symplectic fibrations can be turned into convex contact boundaries
by adding a collar to the boundary carrying a suitable
symplectic form. Cap off all boundary components
with spherical fibres. Then we arrive at a contradiction
as in Section~\ref{subsection:contradiction}.
\section{Proof of Theorem~\ref{thm:ex} (a)}
Theorem~\ref{thm:ex}~(a) will be an obvious consequence of the two lemmata
we prove in this section, together with the transitivity of the cobordism
relation.

\begin{lem}
\label{lem:g-g}
Any two symplectic fibrations $(V^i,\omega^i)$, $i=0,1$,
of the same genus~$g$ (including the case $g=0$) and the same fibre area
are cobordant.
\end{lem}

In particular, given any symplectic fibration of genus~$g$,
there is a cobordism both from and to a trivial symplectic fibration
$\Sigma_g\times S^1$ with symplectic form pulled back from
an area form on $\Sigma_g$ of the appropriate total area.

\begin{proof}[Proof of Lemma~\ref{lem:g-g}]
By \cite[Theorem~3.1]{elia04} there are
compact symplectic $4$-manifolds $(W^i,\Omega^i)$, $i=0,1$, with
$\partial(W^1,\Omega^1)=(V^1,\omega^1)$ and
$\partial(W^0,\Omega^0)=(-V^0,\omega^0)$.  Each
$W^i$ contains a copy of $\Sigma_g$
embedded symplectically in the interior, with trivial normal bundle,
and both of the same area. Hence, we can perform a
fibre connected sum to produce the desired cobordism.
\end{proof}

\begin{lem}
\label{lem:trivial}
For any $g,g'>0$, there is a cobordism between any trivial
symplectic fibration of genus $g$ to any such fibration of genus~$g'$.
\end{lem}

Changing the orientation of the $S^1$-factor in the
trivial symplectic fibration defines an orientation-reversing diffeomorphism
of this fibration to itself, so we can ignore issues of orientation
in the following proof.

\begin{proof}[Proof of Lemma~\ref{lem:trivial}]
Let $(V_g=\Sigma_g\times S^1,\omega_g)$ and
$(V_{g'}=\Sigma_{g'}\times S^1,\omega_{g'})$ be two trivial
symplectic fibrations, where we think of $\omega_g$ as both
the positive area form on $\Sigma_g$ and the $2$-form on $V_g$,
likewise for $\omega_{g'}$.

Choose an area form $\omega_{T^2}$ on $T^2\setminus\Int(D^2)$,
the $2$-torus with an open disc removed. Equip
$\Sigma_g\times\bigl(T^2\setminus\Int(D^2)\bigr)$
with the product symplectic form $\omega_g\oplus\omega_{T^2}$.
Now let $T=a\times b$ be a Lagrangian torus in this symplectic manifold
as in the proof of Proposition~\ref{prop:X}, i.e.\ choose
a simple closed curve $a$ in $\Sigma_g$ representing
a generator of $H_1(\Sigma_g)$, and a simple closed
curve $b$ in $T^2\setminus\Int(D^2)$ representing a generator of the
relative homology group $H_1(T^2\setminus\Int(D^2),\partial D^2)$.
As before we now appeal to \cite[Lemma~1.6]{gomp95}. Since that lemma
is formulated for closed manifolds only, we are a little more
explicit. There are closed $1$-forms $\alpha,\beta$ supported
near the dual curve of $a,b$ on $\Sigma_g, T^2\setminus\Int(D^2)$,
respectively, with
\[ \int_a\alpha=\int_b\beta=1,\;\;\text{hence}\;\;\int_T\alpha\wedge\beta=1.\]
We think of $\alpha$ and $\beta$ as $1$-forms on the symplectic
$4$-manifold by pulling them back under the projection to
one of the two factors.

Let $\eta$ be any area form on $T$ of total area~$1$. With $j$
denoting the inclusion of $T$ in the symplectic $4$-manifold,
the $2$-form $\eta-j^*(\alpha\wedge\beta)$ integrates to zero
over $T$ and hence equals an exact $2$-form $\rmd\gamma$ on~$T$.
Extend $\gamma$ to a $1$-form on the whole $4$-manifold,
supported near~$T$. Then, for any small $\varepsilon>0$,
the $2$-form $\omega_g\oplus\omega_{T^2}+\varepsilon(\alpha\wedge\beta
+\rmd\gamma)$ will be symplectic, and it pulls back to $\varepsilon\eta$
on~$T$, which makes that torus symplectic.

Now perform the same construction starting from $(\Sigma_{g'},\omega_{g'})$.
If we choose the same $\varepsilon$ in both instances, we can then
perform a fibre connected sum along the respective copies of $T$.
This results in the desired cobordism.
\end{proof}
\section{Proof of Theorem~\ref{thm:ex} (b)}
We now discuss the existence of cobordisms between
two symplectic fibrations of genus zero. The following proposition
is a special case of Lemma~\ref{lem:g-g}, but we include a direct
proof, since the case $g=0$ requires considerably less machinery.

\begin{prop}
There is a cobordism between any two spherical symplectic fibrations
of the same fibre area.
\end{prop}

\begin{proof}
The total space of any spherical symplectic fibration
is $S^2\times S^1$.  Any two product symplectic fibrations on this
space, where the $2$-form is pulled back from an area form
on~$S^2$, are diffeomorphic by the usual
Moser argument, provided they have the same total area on~$S^2$.

By the transitivity of the cobordism relation, it
suffices to show that given any spherical symplectic
fibration $(S^2\times S^1,\omega)$, we can find
cobordisms both to and from a product symplectic fibration
with the same fibre area.

The construction of such cobordisms is essentially given in the
proof of \cite[Lemma~3.5]{elia04}. Define $S^2_0=S^2\times\{0\}$,
where we identify $S^1$ with $\R/2\pi\Z$, and set
$\omega_0=\omega|_{TS^2_0}$. Since all symplectomorphisms
of $S^2$ are Hamiltonian, the holonomy of $(S^2\times S^1,\omega)$,
regarded as a symplectomorphism of $(S^2_0,\omega_0)$,
is given as the time-$2\pi$ flow of a $2\pi$-periodic time-dependent
Hamiltonian $H_t\co S^2_0\rightarrow\R$; the corresponding time-dependent
Hamiltonian vector field $X_t$ is given by $i_{X_t}\omega_0=-\rmd H_t$.
We may assume that there
are constants $m$ and $M$ with $0<m<H_t<M$. Define
an embedding $f\co S^2_0\times S^1\rightarrow S^2_0\times\C$ by
\[ (x,t)\longmapsto \bigl(x,\sqrt{H_t(x)}\,\rme^{-\rmi t}\bigr).\]
Then the split form $\Omega_0=\omega_0+\rmd(r^2\rmd\varphi)$
pulls back to~$\omega$. Indeed, we have
\[ f^*\Omega_0=\omega_0-\rmd\bigl(H_t(x)\,\rmd t\bigr)=
\omega_0+i_{X_t}\omega_0\wedge\rmd t,\]
which implies that the characteristic foliation of $f^*\Omega_0$ is
spanned by $\partial_t+X_t$, as desired.

It follows that the restriction of $\Omega_0$ to
\[\bigl\{ (x,z)\in S^2_0\times\C\co\sqrt{m}\leq|z|\leq\sqrt{H_t(x)}\bigr\}\]
and
\[\bigl\{ (x,z)\in S^2_0\times\C\co\sqrt{H_t(x)}\leq|z|\leq\sqrt{M}\bigr\}\]
defines a cobordism between $(S^2\times S^1,\omega)$ and
a product symplectic fibration, in one or the other direction.
\end{proof}

It remains to show that the fibre area is a cobordism invariant.
Thus, suppose we have a cobordism between two spherical symplectic
fibrations. Pick a fibre $S,S'$ in each of the two.
By capping off the cobordism we obtain a closed symplectic $4$-manifold
$(\wtW,\wtOm)$ containing two symplectic spheres $S,S'$
of self-intersection zero. We may assume that
all exceptional spheres in the complement of $S\cup S'$
have been blown down. If each of the remaining exceptional
spheres intersects $S$, then the pair
$(\wtW,S)$ is minimal. We can then choose
a compatible almost complex structure on $\wtW$ for which
both $S$ and $S'$ are holomorphic, and which is regular for
the class~$[S]$. The classification
result of McDuff~\cite{mcdu90}, applied to the pair $(\wtW,S)$
then tells us that $S$ is a fibre in a holomorphic fibration,
and by positivity of intersection $S'$ must coincide with
one of the fibres, so the areas of $S$ and $S'$ are equal.

We conclude the proof by demonstrating that indeed every exceptional
sphere that intersects $S'$ must also intersect $S$. Here we argue
as in the light cone lemma of~\cite{lili95}. According to
McDuff~\cite{mcdu90}, cf.\ \cite[Theorem~A]{wend14}, $(\wtW,\wtOm)$ is
symplectomorphic to a blow-up of $\CP^2$ (with a rescaled Fubini--Study
form) or a blow-up of a symplectically ruled surface. In either case
$\wtW$ has $b_2^+=1$, so we can choose a basis $a,a_1,\ldots,a_n$
of $H_2(\wtW;\R)$ for which the intersection
product is diagonal with $a^2=1$ and $a_i^2=-1$ for $i=1,\ldots,n$.
Write
\[ [S] = \lambda a+\sum_{i=1}^n\lambda_ia_i,\;\;\;\;
[S'] = \mu a+\sum_{i=1}^n\mu_ia_i.\]
Since $S$ and $S'$ are disjoint and of self-intersection~$0$, we have
\[ \lambda\mu-\sum_i\lambda_i\mu_i=0,\;\;\;\;
\lambda^2-\sum_i\lambda_i^2=0,\;\;\;\;
\mu^2-\sum_i\mu_i^2=0.\]
The Cauchy--Schwarz inequality gives
\[ |\lambda\mu|=\sqrt{\sum_i\lambda_i^2}\sqrt{\sum_i\mu_i^2}\geq
\left|\sum_i\lambda_i\mu_i\right|=|\lambda\mu|.\]
So here we must have equality, and hence
\[ (\lambda_1,\ldots,\lambda_n) = r(\mu_1,\ldots,\mu_n)\]
for some $r\in\R$. It follows further that $\lambda=\pm r\mu$.
From the positivity of $\bigl\langle\wtOm,[S]\bigr\rangle$
and $\bigl\langle\wtOm,[S']\bigr\rangle$ we see that
the negative sign is impossible, since the equations
coming from the intersection products would then imply that both
$[S']$ and $[S]$ are the zero class. Moreover, we must have $r>0$, i.e.\
the class $[S]$ is a positive multiple of~$[S']$.

As in Section~\ref{subsection:alt} we may assume that $S,S'$ and all
exceptional spheres are holomorphic. By positivity of intersection
we conclude that for any exceptional sphere $E$ that intersects $S'$
the intersection product $S'\bullet E$ is positive. Hence so is
$S\bullet E$, as we wanted to show.
\begin{ack}
We are grateful to \.Inan\c{c} Baykur for his comments on a draft
version of this paper, and to Fran Presas for the alternative proof
of Theorem~\ref{thm:non-ex} outlined in Section~\ref{subsection:confol}.
We also thank an anonymous referee for
suggesting the argument that establishes the fibre area as
a cobordism invariant of spherical symplectic fibrations.
\end{ack}

\end{document}